\documentclass[12pt, a4paper, reqno]{amsart}
\usepackage{amsmath}
\usepackage{amsfonts}
\usepackage{amssymb}
\usepackage{amsthm}
\usepackage{url}

\newcommand{\R}{\mathbb{R}}

\DeclareMathOperator*{\osc}{osc}

\def\vint_#1{\mathchoice%
          {\mathop{\kern 0.2em\vrule width 0.6em height 0.69678ex depth -0.58065ex
                  \kern -0.8em \intop}\nolimits_{\kern -0.4em#1}}%
          {\mathop{\kern 0.1em\vrule width 0.5em height 0.69678ex depth -0.60387ex
                  \kern -0.6em \intop}\nolimits_{#1}}%
          {\mathop{\kern 0.1em\vrule width 0.5em height 0.69678ex depth -0.60387ex
                  \kern -0.6em \intop}\nolimits_{#1}}%
          {\mathop{\kern 0.1em\vrule width 0.5em height 0.69678ex depth -0.60387ex
                  \kern -0.6em \intop}\nolimits_{#1}}}

\newcommand{\art}[6]{{\sc #1, \rm #2, \it #3 \bf #4 \rm (#5), \mbox{#6}.}}
\newcommand{\book}[3]{{\sc #1, \it #2, \rm #3.}}
\newcommand{\AND}{{\rm and }}
\newcommand{\p}{{$p\mspace{1mu}$}}

\newcommand{\loc}{_{\rm loc}}

\newcommand{\eps}{\varepsilon}

\theoremstyle{plain}
\newtheorem{theorem}[equation]{Theorem}
\newtheorem{lemma}[equation]{Lemma}
\newtheorem{corollary}[equation]{Corollary}

\numberwithin{equation}{section}

\theoremstyle{definition}

\theoremstyle{remark}
\newtheorem{remark}[equation]{Remark}
\title{On a frequency function approach to the unique
  continuation principle}

\author{Seppo Granlund}
\address[S.G.]{University of Helsinki,
  Department of Mathematics and Statistics,
  P.O. Box 68, FI-00014 University of Helsinki, Finland}
\email{seppo.granlund@pp.inet.fi}

\author{Niko Marola}
\address[N.M.]{University of Helsinki,
  Department of Mathematics and Statistics,
  P.O. Box 68, FI-00014 University of Helsinki, Finland}
\email{niko.marola@helsinki.fi}

\date{}

\begin{document}



\begin{abstract}
  In this survey we discuss the frequency function method so as to
  study the problem of unique continuation for elliptic partial
  differential equations. The methods used in the note were mainly
  introduced by Garofalo and Lin.
\end{abstract}

\maketitle

\section{Introduction}

Let $G$ be an open connected subset of $\R^n$, $n\geq 2$. We consider
the problem of unique continuation for both the solutions to the
Laplace equation and to equation
\begin{equation} \label{eq:Pde}
\Delta u = b(x)\cdot\nabla u, 
\end{equation}
where the drift coefficients, $\{b_i(x)\}_{i=1}^n$, are
continuous and bounded in $G$. The classical unique continuation
principle for the latter equation can be formulated as follows
\begin{itemize}

\item[(i)] Let $u_1$ and $u_2$ be two solutions to \eqref{eq:Pde} such
  that $u_1=u_2$ in an open subset of $G$. Then $u_1\equiv u_2$ in
  $G$.

\item[(ii)] Let $u$ be a solution to \eqref{eq:Pde} such that $u=0$
  in an open subset of $G$. Then $u\equiv 0$ in $G$.

\end{itemize}
The latter formulation is equivalent to the following: (ii') Let $u$
be a solution to \eqref{eq:Pde} and consider two open concentric
balls $B_r\subset \overline{B}_R\subset G$ such that $u=0$ on $B_r$,
then $u\equiv 0$ in $B_R$.

Instead of using Carleman's method to deal with the unique
continuaton, we follow the method introduced by Garofalo and Lin in
\cite{GarLinIndiana} and \cite{GarLinCPAM}, see also Fabes et
al.~\cite{FaGaLin}. Their method is based on the ingenious analysis of
(a modification of) Almgren's frequency function, see \cite{Almgren},
which, in turn, leads to monotonicity formulas and doubling
inequalities. The main result in \cite{GarLinCPAM} is the unique
continuation principle for the solutions to the equation
\begin{equation} \label{eq:GarLin}
-\nabla\cdot(A(x)\nabla u) + b(x)\cdot\nabla u + V(x)u=0,
\end{equation}
where $A(x) = (a_{ij}(x))_{i,j=1}^n$ is a real symmetric matrix-valued
function satisfying the uniform ellipticity condition and it is
Lipschitz continuous. The lower order terms, the drift coefficient
$b(x)$ and the potential $V(x)$, are even allowed to have
singularities.  The reader should consult (1.4)--(1.6) in
\cite{GarLinCPAM} for the exact structure conditions of $b$ and $V$.

In the present survey, our goal is to provide a clear user's
guide-type presentation on this topic, and we do not attempt to deal
with the most general case \eqref{eq:GarLin}. For such a treatise, the
reader should consult more advanced papers \cite{GarLinIndiana} and
\cite{GarLinCPAM}, and a paper \cite{TaoZhang} by Tao and Zhang.

Our proofs are by contradiction, which makes it possible to use
Poincar\'e's inequality in certain phases of the proof. By this
observation we are able to obtain more straightforward treatment for
the classical proof, however our method is indirect.

In outline, a brief discussion on the Rellich--Necas identity, as well
as the notation, can be found in \textsection~\ref{sect:Prel}. The
unique continuation principle for the Laplace equation is covered in
\textsection~\ref{sect:Laplace}, and for the solutions to
\eqref{eq:Pde} in \textsection~\ref{sect:Pde}. We close this note by
discussing possible generalizations to the nonlinear case in
\textsection~\ref{sect:Gener}, i.e., unique continuation principle for
the \p-Laplace equation, $$\nabla\cdot(|\nabla u|^{p-2}\nabla u)= 0,$$
where $1<p<\infty$. Observe that in the case $p=2$ we recover the
Laplace equation. We do not claim that the frequency function method
is a panacea for the unique continuation principle in this nonlinear
case, but it seems to open new possibilities to study the problem.

We want to remark that the unique continuation for the solutions to
\eqref{eq:Pde} is interestingly entwined with the one for the
\p-Laplace equation (see \textsection~\ref{sect:Gener}). Lastly, in
contrast to the Laplace equation, equation \eqref{eq:Pde} is more
subtle and to reach the unique continuation principle for its
solutions a great deal more analysis is required.

\section{Preliminaries}
\label{sect:Prel}

Throughout the present note, $G$ is open and connected subset of
$\R^n$, $n\geq 2$. We use the notation $B_r = B(x,r)$ for concentric
open balls of radii $r$ centered at $x\in G$. Unless otherwise stated,
the letter $C$ denotes various positive and finite constants whose
exact values are unimportant and may vary from line to line. Moreover,
$dx=dx_1\ldots dx_n$ denotes the Lebesgue volume element in $\R^n$,
whereas $dS$ denotes the surface element. We denote by $|E|$ the
n-dimensional Lebesgue measure of a measurable set
$E\subseteq\R^n$. Along $\partial G$, whenever $G$ is smooth enough,
is defined the outward pointing unit normal vector field at
$x\in \partial G$ and is denoted by
$\nu(x)=(\nu_1,\ldots,\nu_n)(x)$. We will also write $u_\nu = \nabla
u\cdot \nu$ or $\partial u/\partial \nu$ for the directional
derivative of $u$. We denote a tangential gradient by $\nabla_t$.

\medskip

We shall make use of the following Rellich--Necas type identity. To
the best of our knowledge, this formula was first employed by Payne
and Weinberger in \cite{PaWe}, and it is a variant of a formula due to
Rellich~\cite{Rellich} and Necas. We also refer to
Jerison--Kenig~\cite{JeKe}. A Rellich--Necas type formula appears,
e.g., in Pucci--Serrin~\cite{PuSe}, Garofalo--Lewis~\cite{GaLe}, and
Lewis--Vogel~\cite{LeVo}.

\begin{lemma} \label{Rellic-Necas} Let $u\in C^2(G)\cap C^1(\overline
  G)$. The following formula is valid
  \begin{align} \label{eq:RN} -\int_{G} & \left(2(x\cdot\nabla
      u)\Delta
      u+(n-2)u\Delta u\right)\, dx \nonumber \\
    & = \int_{\partial G}\left(|\nabla u|^2(x\cdot\nu)-2(x\cdot\nabla
      u)u_\nu - (n-2)uu_\nu\right)\, dS.
\end{align}
In particular, if $u$ is harmonic in $G$ then \eqref{eq:RN} reduces to
the following formula
  \begin{equation} \label{eq:RN-harm} \int_{\partial G}\left(|\nabla
      u|^2(x\cdot\nu)-2(x\cdot\nabla u)u_\nu - (n-2)uu_\nu\right)\, dS
    = 0
\end{equation}
\end{lemma}

\begin{proof}
  The proof follows from the following divergence identity which stems
  from Noether's theorem; observe (\cite[eq. 3.2]{PaWe}, and also
  \cite[p. 204]{JeKe}), by a direct calculation, that
\begin{align} \label{eq:RNdiv}
  \nabla\cdot & \left(|\nabla u|^2x-2(x\cdot\nabla
    u)\nabla u -(n-2)u\nabla u \right) \nonumber \\
  & = -2(x\cdot\nabla u)\Delta u-(n-2)u\Delta u.
\end{align}
Then integrating over $G$ and applying the Gauss theorem we arrive at
\eqref{eq:RN}. Equation \eqref{eq:RN-harm} follows from \eqref{eq:RN}
simply by setting $\Delta u=0$.
\end{proof}

\begin{remark}
  We remark that using the fact that $$|\nabla u|^2=|\nabla_t
  u|^2+|u_\nu|^2$$ and denoting $\alpha(x)=x-(x\cdot\nu)\nu$ we may rewrite
  \eqref{eq:RN-harm} as follows
\begin{align*}
\int_{\partial G} & \left(\left(|\nabla_t
    u|^2-|u_\nu|^2\right)(x\cdot\nu) +2(\alpha(x)\cdot\nabla u)u_\nu\right. \\
& \qquad  - (n-2)uu_\nu\big)\, dS = 0,
\end{align*}
which is just equation (2) in Jerison--Kenig~\cite{JeKe}.
\end{remark}

For harmonic functions and for each $B_r\subset G$, $x\in \partial
B_r$, $\nu$ is the outward pointing unit normal at $x$, we may extract
from \eqref{eq:RN-harm}, or from \eqref{eq:RNdiv}, the following
\begin{align} \label{eq:RNball} r\int_{\partial B_r}|\nabla u|^2\, dS
  & = 2r\int_{\partial B_r}|u_\nu|^2\, dS + (n-2)\int_{\partial
    B_r}uu_\nu\, dS \nonumber \\
  & = 2r\int_{\partial B_r}|u_\nu|^2\, dS + (n-2)\int_{B_r}|\nabla
  u|^2\, dx.
\end{align}

Equivalently, \eqref{eq:RNball} may be stated as a Hardt--Lin
\cite[Lemma 4.1]{HardtLin} type monotonicity identity
\[
\frac{d}{dr}\left(r^{2-n}\int_{B_r}|\nabla u|^2\, dx\right) =
2r^{2-n}\int_{\partial B_r}|u_\nu|^2\, dS.
\]

\medskip

We shall also need the following Poincar\'e type inequality, consult
Giusti~\cite{Giusti} for the proof.  Suppose $u\in W^{1,2}(B_r)$ and
let $Z=\{x\in B_r: u(x)=0\}$. If there exists a constant $0<\gamma<1$
such that $|Z| \geq \gamma |B_r|$, then there exists a constant $C_p$,
depending on $n$ and $\gamma$, such that
\begin{equation} \label{eq:Poincare}
\int_{B_r}u^2\, dx \leq C_pr^2\int_{B_r}|\nabla u|^2\, dx.
\end{equation}

\section{Unique continuation: Laplace equation}
\label{sect:Laplace}

Almgren's \cite{Almgren} insight was that for a harmonic function $u$
the function
\begin{equation} \label{FF1}
F(r) = \frac{r\int_{B_r}|\nabla u|^2\, dx}{\int_{\partial B_r}u^2\, dS},
\end{equation}
called the \emph{frequency function}, is monotonically non-decreasing
as a function of $r$. He observed, moreover, that by employing this
property one is able to deduce the unique continuation principle for
the solutions to the Laplace equation. See \cite{Almgren} for more
properties of the frequency function.

In what follows, we denote the numerator by $rD(r)$ and the
denominator by $I(r)$. The following is, of course, well-known but we
treat it here since the proof is rather short and simple. Let us
demonstrate how the result is reached.

\begin{theorem} \label{thm:UC1}
  Suppose $u\in C^2(G)$ and $\Delta u=0$ in $G$. If there is an open
  set $D\subset G$ such that $u=0$ in $D$, then $u\equiv 0$ in $G$.
\end{theorem}

\begin{proof}
  We prove the following from which the claim follows easily: Assume
  $0<r_1<r_2$ and $B_{r_1}\subset\overline B_{r_2}\subset G$. If
  $u(x)=0$ in $B_{r_1}$, then $u(x)=0$ in $B_{r_2}$. To prove this, we
  assume, on the contrary, that there exists $x_0\in G$ so that
  $u(x)=0$ in $B_{r_1}(x_0)$ but $u$ is not identically zero in
  $B_{r_2}(x_0)$. It will be shown below that function $I(r)$ is
  non-decreasing. Then $I(r_2)>0$ and there is a number
  $r_0\in[r_1,r_2]$ such that $I(r_0)=0$, but $I(s)>0$ for $s>r_0$. We
  thus consider an interval $[s,r_2]$, where $s>r_0$.

  Let us start by proving a Harnack type inequality for $I(r)$. Since
\begin{equation} \label{eq:I'(r)}
I'(r) = \frac{n-1}{r}I(r) +  2\int_{\partial B_r}uu_\nu\, dS,
\end{equation}
it follows from the following Gauss--Green identity,
\[
\int_{B_r}|\nabla u|^2\, dx + \int_{B_r}u\Delta u\, dx =
\int_{\partial B_r}uu_\nu\, dS,
\]
that
\[
\int_{\partial B_r}uu_\nu\, dS = \int_{B_r}|\nabla u|^2\, dx \geq 0.
\]
Hence $r\mapsto I(r)$ is non-decreasing. We will consider the function
$H(r) = \log I(r)$, which is also non-decreasing. The derivative of
$H(r)$ is
\begin{equation} \label{eq:H'(r)}
H'(r) = \frac{n-1}{r} + \frac{2F(r)}{r}.
\end{equation}
We use \eqref{eq:H'(r)} to obtain an upper bound for the oscillation
of $H(r)$ on $[s,t]\subset [s,r_2]$ as follows
\begin{align} \label{eq:oscH}
  \osc_{r\in [s,t]}H(r) & =\max_{r\in [s,t]} H(r) -\min_{r\in [s,t]} H(r) \nonumber \\
  & = H(t)-H(s) = \int_s^tH'(r)\, dr  \nonumber \\
  & =  \int_s^t\left(\frac{n-1}{r} + \frac{2F(r)}{r}\right)\, dr \nonumber \\
  & \leq
  \left(n-1+2\sup_{r\in[s,t]}F(r)\right)\log\left(\frac{t}{s}\right).
\end{align}
From \eqref{eq:oscH} it follows that
\[
\frac{\max_{r\in [s,t]} I(r)}{\min_{r\in [s,t]} I(r)} \leq \left(\frac{t}{s}\right)^{n-1+2\sup_{r\in[s,t]}F(r)},
\]
which implies the following Harnack type inequality
\begin{equation} \label{eq:Harnack}
\max_{r\in [s,t]} I(r) \leq \left(\frac{t}{s}\right)^{n-1+2\sup_{r\in[s,t]}F(r)}\min_{r\in [s,t]} I(r).
\end{equation}
The next step is to show that Almgren's frequency function is
non-decreasing.  The derivative of $F(r)$ is 
\begin{equation} \label{eq:F'(r)}
F'(r) = \frac{D(r)I(r)+rD'(r)I(r)-rD(r)I'(r)}{I^2(r)},
\end{equation}
where $D'(r) = \int_{\partial B_r}|\nabla u|^2\, dS$. From the
Rellich--Necas type identity \eqref{eq:RNball} we obtain
\begin{align} \label{eq:rD'I}
  r & D'(r)I(r) = \left(\int_{\partial B_r}u^2\, dS\right)\left(r\int_{\partial B_r}|\nabla u|^2\, dS\right) \nonumber \\
  & = \left(\int_{\partial B_r}u^2\, dS\right)\left(2r\int_{\partial B_r}|u_\nu|^2\, dS + (n-2)D(r)\right) \nonumber \\
  & = 2r \left(\int_{\partial B_r}u^2\,
    dS\right)\left(\int_{\partial B_r}|u_\nu|^2\, dS\right) + (n-2)D(r)I(r). 
\end{align}
Plugging \eqref{eq:rD'I} and \eqref{eq:I'(r)} into \eqref{eq:F'(r)} we
arrive at 
\begin{align*}
  I^2(r)F'(r) & = 2r \left(\int_{\partial B_r}u^2\,
    dS\right)\left(\int_{\partial B_r}|u_\nu|^2\, dS\right)-2r\left(\int_{\partial B_r}uu_\nu\, dS\right)^2 \\
  & \geq 2r \left(\int_{\partial B_r}uu_\nu\,
    dS\right)^2-2r\left(\int_{\partial B_r}uu_\nu\, dS\right)^2 = 0,
\end{align*}
where we used H\"older's inequality. It follows that $F(r)$ is
non-decreasing, and hence we may control the exponent in
\eqref{eq:Harnack} from above.

To finish the proof, from \eqref{eq:Harnack} we obtain
\[
I(t) = \max_{r\in [s,t]} I(r) \leq \left(\frac{t}{s}\right)^{n-1+2F(t)}I(s).
\]
Since $I(s)\to 0$ as $s\to r_0$, it follows that $I(t)=0$. This is a
contradiction.
\end{proof}

We state the following immediate corollary (of the preceding proof) as
it migth be of independent interest to the reader.

\begin{corollary}
  Let $u\in C^2(G)$ and $\Delta u=0$ in $G$. Suppose that 
\[
I(r)= \int_{\partial B_r}u^2\, dS > 0
\]
at every $r\in (s,t)$, $B_s\subset \overline{B}_t\subset G$. Then the
following Harnack type inequality is valid
\begin{equation} \label{eq:HarnackforI}
\max_{r\in (s,t)}\vint_{\partial B_r}u^2\, dS \leq
\left(\frac{t}{s}\right)^{2F(t)}\min_{r\in (s,t)}\vint_{\partial
  B_r}u^2\, dS.
\end{equation}
\end{corollary}

Note that for \eqref{eq:HarnackforI} one needs to observe that
$r^{1-n}I(r)$ is non-decreasing, and then the inequality follows
immediately from \eqref{eq:Harnack}.

We may also estimate how rapidly a harmonic function grows near a
point where it vanishes. Namely, it is well-known but noteworthy that
from the fact that $F(r)$ is non-decreasing it directly follows from
\eqref{eq:HarnackforI} that for $0<r<R$
\[
\int_{\partial B_r}u^2\, dS \geq \gamma r^{\beta+n-1},
\]
where $\gamma := I(R)R^{-\beta-n+1}$ and $\beta := 2F(R)$.

\section{Unique continuation: $\Delta u = b (x)\cdot \nabla u$}
\label{sect:Pde}

We shall deal with the following modified version of Almgren's
frequency function
\begin{equation} \label{FF2}
F(r) = \frac{r\int_{\partial B_r}uu_\nu\, dx}{\int_{\partial B_r}u^2\,
  dS},
\end{equation}
and denote the numerator by $rH(r)$ and the denominator by $I(r)$. Of
course, for harmonic functions \eqref{FF2} is equal to \eqref{FF1}
thanks to the Gauss--Green identity. It is important to note that the
frequency function defined in \eqref{FF2} is not necessarily
non-negative for all radii $r>0$.

One may easily check that the frequency function defined in
\eqref{FF2}, as well as in \eqref{FF1}, is invariant under scaling in
the following sense: Let $\tau \in \R$, $\tau>0$, and denote
$v(x)=u(\tau x)$, where $u$ is a solution to \eqref{eq:Pde}. Then
\[
F^v(r)=F^{u}(\tau r)
\]
for each $r>0$, where $F^v(r)$ denotes the frequency function
associated with function $v$.

The theorem we prove is the following. The proof is an extension of
the harmonic case presented in the preceding section yet more subtle
and demanding.

\begin{theorem} \label{thm:UC2} Suppose $u\in C^2(G)$ is a solution
  to $$\Delta u = b(x)\cdot \nabla u$$ in $G$, where the drift
  coefficients $\{b_i(x)\}_{i=1}^n$ are continuous and bounded in
  $G$. If there is an open set $D\subset G$ such that $u=0$ in $D$,
  then $u\equiv 0$ in $G$.
\end{theorem}

As opposed to Almgren's frequency function, \eqref{FF1}, frequency
function for solutions to \eqref{eq:Pde} as defined in \eqref{FF2} is
not known to be non-decreasing in $r$. To overcome this,
the key idea is to obtain the following inequality
\begin{equation} \label{eq:F'ineq}
F'(r) \geq -\frac{\alpha}{r}(F(r) + \beta),
\end{equation}
where $0<\alpha,\beta<\infty$ are not depending on $r$. Inequality
\eqref{eq:F'ineq} is obtained only for small values of $r$. Then
setting $T(r) := F(r)+\beta$, and thus $T'(r) = F'(r)$, we may rewrite
\eqref{eq:F'ineq} as follows
\begin{equation} \label{eq:logT}
\frac{d}{dr}\log T(r) \geq -\frac{\alpha}{r}.
\end{equation}
From \eqref{eq:logT} one may deduce the following for each pair
$r<\rho$
%
%
\[
T(r)\leq \left(\frac{\rho}{r}\right)^{\alpha}T(\rho),
\]
i.e., 
\begin{equation} \label{eq:Fineq} F(r)\leq
  \left(\frac{\rho}{r}\right)^{\alpha}F(\rho)+\beta\left(\left(\frac{\rho}{r}\right)^{\alpha}-1\right).
\end{equation}

The detailed proof below is rather technical, but straightforward.

\begin{proof}[Proof of Theorem~\ref{thm:UC2}]
  As in the proof of Theorem~\ref{thm:UC1}, the proof is by
  contradiction. Suppose, on the contrary, that there is an open set
  $D\subset G$ such that $u=0$ in $D$, but $u$ is not identically zero
  in $G$. Then it is possible to pick arbitrary small neighborhoods
  $B_{r_1}(x_0)$ and $B_{r_2}(x_0)$,
  $\overline{B}_{r_1}(x_0),\,\overline{B}_{r_2}(x_0)\subset G$, such
  that $u(x)=0$ in $B_{r_1}(x_0)$ but $u$ is not identically zero in
  $B_{r_2}(x_0)$. This can be shown by connecting a point $x_1\in D$
  to a point $x_2\in G\setminus D$ such that $u(x_2)\neq 0$, by a rectifiable
  curve in $G$, taking a finite sub-cover of balls with arbitrary
  small radii, and by employing a well-known chaining
  argument. Observe further that radii $r_1$ and $r_2$, which are to
  be fixed later, can be chosen in such a way that there exists
  $0<\gamma_0<1$ so that
  \[
  \frac{|B_{r_1}(x_0)|}{|B_{r_2}(x_0)|}\geq \gamma_0.
  \]
  This enables us to employ Poincar\'e's inequality. 

  In order to show that $I(r)$ is non-decreasing for small values of
  $r$, we start by showing that there exists $r_2>0$ such that
  $H(r)\geq 0$ for each $0<r\leq r_2$. By the Poincar\'e inequality,
  \eqref{eq:Poincare}, we get
\begin{align*}
  \left|\int_{B_r}u\Delta u\, dx\right| & \leq \left(\int_{B_r}u^2\, dx\right)^{1/2}\left(\int_{B_r}\left|b(x)\cdot\nabla u\right|^2\, dx\right)^{1/2} \\
  & \leq \sqrt{C_p}M\left(r^2\int_{B_r}|\nabla u|^2\, dx\right)^{1/2}\left(\int_{B_r}|\nabla u|^2\, dx\right)^{1/2} \\
  & = \sqrt{C_p}Mr\int_{B_r}|\nabla u|^2\, dx,
\end{align*}
where $M:=\|b\|_{L^\infty(G)}<\infty$ and $C_p$ is the constant in the
Poincar\'e inequality and here it depends on $\gamma_0$. We now select
$r_2$ small enough so that $\sqrt{C_p}Mr < 1/2$ for every $r\leq
r_2$. Plugging the preceding estimate into the Gauss--Green formula we
arrive at
\begin{equation} \label{eq:DH}
\int_{B_r}|\nabla u|^2\, dx \leq  2\int_{\partial B_r}uu_\nu\, dS = 2H(r),
\end{equation}
and hence $H(r)\geq 0$ for every $0<r\leq r_2$. In addition, we easily see
that $I(r)$ is non-decreasing on $(0,r_2)$ as
\begin{equation} \label{eq:I'} I'(r) = \frac{n-1}{r}I(r) +
  2\int_{\partial B_r}uu_\nu\, ds = \frac{n-1}{r}I(r) +
  2H(r).
\end{equation}
Since we know that $I(r_2)>0$, there exists a radius $\tilde r
\in[r_1,r_2]$ such that $I(\tilde r)=0$, but $I(r)>0$ for $r>\tilde
r$. From here on out, we thus consider an interval $(\tilde r, r_2]$.

Let us examine the derivative of $F(r)$. We have
\begin{equation} \label{eq:F'}
F'(r) = \frac{H(r)I(r)+rH'(r)I(r)-rH(r)I'(r)}{I^2(r)}.
\end{equation}
On the other hand, we obtain again from the Gauss-Green formula that
\begin{equation} \label{eq:H'}
H'(r) = \int_{\partial B_r}|\nabla u|^2\, dS + \int_{\partial B_r}u\Delta u\, dS.
\end{equation}
Plugging \eqref{eq:H'} into \eqref{eq:F'} and using \eqref{eq:I'} we
have the following expression for the derivative of the frequency
function
\begin{align} \label{eq:IF'} I^2(r)F'(r) & = H(r)I(r)+rI(r)\int_{\partial
    B_r}|\nabla u|^2\, dS + rI(r)\int_{\partial B_r}u\Delta u\, dS
  \nonumber \\
& \qquad - (n-1)H(r)I(r) - 2rH^2(r).
\end{align}
At this point, we distinguish the following two possibilities. This is
one of the crucial points in the proof of this theorem, and is in many
ways analogous to Case 1 and 2, i.e., (2.49) and (2.51) in Garofalo
and Lin~\cite{GarLinCPAM}. As it will become clear, out of the two
cases (B) is much stronger.

\begin{enumerate}

\item[] 

\item[(A)] $I(r)\int_{\partial B_r}|\nabla u|^2\, dx\leq 4H^2(r)$;

\item[]

\item[(B)] $I(r)\int_{\partial B_r}|\nabla u|^2\, dx > 4H^2(r)$.

\item[]

\end{enumerate}
Clearly either (A) or (B) holds true. 

Suppose first that (A) is valid. We continue by estimating the terms
on the right in \eqref{eq:IF'}. The third term can be estimated as
follows using equation \eqref{eq:Pde} and hypothesis (A)
\begin{align} \label{eq:DeltauA}
  \left|\int_{\partial B_r}u\Delta u\, dS\right| & \leq \left|\int_{\partial B_r}u(b(x)\cdot\nabla u)\, dS\right| \leq M\int_{\partial B_r}|u||\nabla u|\, dS \nonumber \\
  & \leq M\left(\int_{\partial B_r}u^2\,
    dS\right)^{1/2}\left(\int_{\partial B_r}|\nabla u|^2\,
    dS\right)^{1/2} \nonumber \\
  & \leq 2MH(r).
\end{align}
We handle the second term on the right in \eqref{eq:IF'} using
Rellich--Necas type equation \eqref{eq:RN}. We have
\begin{align} \label{eq:IF'2term}
  rI(r) & \int_{\partial B_r}|\nabla u|^2\, dS = 2rI(r)\int_{\partial B_r}|u_\nu|^2\, dS + (n-2)I(r)\int_{\partial B_r}uu_\nu\, dS \nonumber \\
  & - 2I(r)\int_{B_r}(x\cdot\nabla u)\Delta u\, dx -
  (n-2)I(r)\int_{B_r}u\Delta u\, dx.
\end{align}
We note first that by using H\"older's inequality the first term on
the right in \eqref{eq:IF'2term} can be estimated as follows 
\begin{equation} \label{eq:RN1}
I(r)\int_{\partial B_r}|u_\nu|^2\, dS \geq \left(\int_{\partial
    B_r}uu_\nu\, dS\right)^2 = H^2(r).
\end{equation}
Then the last two terms in \eqref{eq:IF'2term} can be controlled
as follows. On one hand, we obtain
\begin{align} \label{eq:RN2}
  \left|\int_{B_r}(x\cdot\nabla u)\Delta u\, dx\right| & \leq r\left(\int_{B_r}|\nabla u|^2\, dx\right)^{1/2}\left(\int_{B_r}|\Delta u|^2\, dx\right)^{1/2} \nonumber \\
  & \leq Mr\int_{B_r}|\nabla u|^2\, dx \leq 2MrH(r)
\end{align}
for every $\tilde r < r\leq r_2$, where we used \eqref{eq:Pde} and
\eqref{eq:DH}. On the other hand, we may estimate as above by using
Poincar\'e inequality \eqref{eq:Poincare} and \eqref{eq:DH}
\begin{align} \label{eq:RN3}
  \left|\int_{B_r}u\Delta u\, dx\right|& \leq \left(\int_{B_r}u^2\, dx\right)^{1/2}\left(\int_{B_r}|\Delta u|^2\, dx\right)^{1/2} \nonumber \\
  & \leq \sqrt{C_p}Mr\int_{B_r}|\nabla u|^2\, dx \leq H(r).
\end{align}
By first plugging \eqref{eq:RN1}, \eqref{eq:RN2}, and \eqref{eq:RN3}
into \eqref{eq:IF'2term}, and then by coupling \eqref{eq:IF'2term}
and \eqref{eq:DeltauA} with \eqref{eq:IF'}, we may continue estimating
\eqref{eq:IF'} using again hypothesis (A) as follows
\begin{align*}
  I^2(r)F'(r) & \geq H(r)I(r) + 2rH^2(r) + (n-2)H(r)I(r) - 4MrH(r)I(r) \\
  & \quad - (n-2)H(r)I(r) - 2MrH(r)I(r) - (n-1)H(r)I(r) \\
  & \qquad - 2rH^2(r) \\
  & \geq -(n-2)H(r)I(r) -6MrH(r)I(r).
\end{align*}
From which we get an inequality of the form \eqref{eq:F'ineq} for $\tilde r<r\leq r_2$ 
\[
F'(r) \geq -\frac{n-2}{r}F(r) - 6MF(r) \geq -\frac{\alpha}{r}F(r),
\]
where $\alpha=n-2+6Mr_2$.

\medskip

Assume now that (B) holds true. 

We estimate the third term on the right in \eqref{eq:IF'}
as follows using equation \eqref{eq:Pde} and the Cauchy
inequality with $\eps=1/(2M)>0$
\begin{align} \label{eq:deltau}
  \left|\int_{\partial B_r}u\Delta u\, dS\right| & \leq \left|\int_{\partial B_r}u(b(x)\cdot\nabla u)\, dS\right| \leq M\int_{\partial B_r}|u||\nabla u|\, dS \nonumber \\
  & \leq 2M^2\int_{\partial B_r}u^2\, dS + \frac1{2}\int_{\partial
    B_r}|\nabla u|^2\, dS.
\end{align}
Then we estimate in \eqref{eq:IF'} using first hypothesis (B) and then
\eqref{eq:deltau} as follows
\begin{align*}
  I^2(r)F'(r) & = H(r)I(r) + rI(r)\int_{\partial B_r}|\nabla u|^2\, dS + rI(r)\int_{\partial B_r}u\Delta u\, dS \\
  & \quad - (n-1)H(r)I(r) - 2rH^2(r) \\
  & \geq 2rH^2(r) + \frac1{2}rI(r)\int_{\partial B_r}|\nabla u|^2\, dS - 2M^2rI^2(r) \\
  & \quad - \frac1{2}rI(r)\int_{\partial B_r}|\nabla u|^2\, dS - (n-1)H(r)I(r) - 2rH^2(r) \\
  & \geq -2M^2rI^2(r) - (n-1)H(r)I(r).
\end{align*}
This implies an inequality of the required form for $\tilde r<r\leq r_2$
\[
F'(r) \geq -\frac{n-1}{r}\left(F(r) + \frac{2(Mr_2)^2}{n-1}\right).
\]

In conclusion, cases both (A) and (B) lead to an inequality of
the form \eqref{eq:F'ineq}.

We may proceed as in the proof of Theorem~\ref{thm:UC1}. In a similar
fashion, we obtain a Harnack type inequality as in \eqref{eq:Harnack},
i.e.,
\[
I(t) \leq
\left(\frac{t}{s}\right)^{n-1+2\sup_{r\in[s,t]}F(r)}I(s)
\]
for $[s,t]\subset (\tilde r,r_2]$, where using \eqref{eq:Fineq} we may
estimate
\begin{equation} \label{eq:Fineqprecise}
\sup_{r\in[s,t]}F(r) \leq \left(\frac{r_2}{\tilde r}\right)^\alpha
F(r_2) + \beta\left(\left(\frac{r_2}{\tilde r}\right)^\alpha-1\right).
\end{equation}
Since $I(s)\to 0$ as $s\to \tilde r$, it follows that $I(t)=0$. This
is a contradiction.
\end{proof}

We remark that by using the frequency function it is possible to
obtain a representation formula for $I(r)$. More precisely, the fact
that
\[
\frac{\tilde I'(r)}{\tilde I(r)}= \frac{2}{r}F(r),
\]
where $\tilde I(r) = r^{1-n}I(r)$, implies the following
\begin{equation} \label{eq:RepresFormulaI}
\int_{\partial B_r}u^2\, dS =
\gamma\exp\left(-2\int_r^RF(t)\frac{dt}{t}\right)r^{n-1}
\end{equation}
for $0<r<R$, where $\gamma:= I(R)$. Equation \eqref{eq:RepresFormulaI}
enables to derive a priori lower bounds for $I(r)$ provided that an
estimate of the form \eqref{eq:Fineqprecise} is available for the
frequency function $F(r)$. Note, however, that the method in the
present paper is by contradiction, and hence we are not able to apply
directly \eqref{eq:Fineqprecise}. A posteriori, it is known that an
estimate like \eqref{eq:Fineqprecise} is valid for the solutions to
\eqref{eq:Pde}, see \cite{GarLinCPAM, TaoZhang}.

\section{Nonlinear generalizations}
\label{sect:Gener}

Consider the \p-Laplace equation in $G$
\begin{equation} \label{eq:pLap} \nabla\cdot(|\nabla u|^{p-2}\nabla
  u)=0, \qquad 1<p<\infty.
\end{equation}
For $p=2$ we recover the Laplace equation $\Delta u=0$. We refer the
reader to, e.g., Heinonen et al.~\cite{HKM} and
Lindqvist~\cite{Lindqvist} for a detailed study of the \p-Laplace
equation and various properties of its solutions. The problem of
unique continuation, both (i) and (ii), is still, to the best of our
knowledge, an open problem, except for the linear case $p=2$. The
planar case for (ii) has been solved by Manfredi in \cite{Manfredi},
see also Bojarski and Iwaniec~\cite{BoIw}, as they have observed that
the complex gradient of a solution to \eqref{eq:pLap} is quasiregular.

In addition to unique continuation, a long-standing open problem is to
find a frequency function associated with solutions to
\eqref{eq:pLap}.

In \cite{GraMa} the authors of the present paper deal with the problem
of unique continuation by studying a certain generalization of
Almgren's frequency function for the \p-Laplacian. By this approach
some partial results on the unique continuation problem in both cases
(i) and (ii) were obtained. Two possible nonlinear generalizations for
the frequency function defined in \cite{GraMa} were as follows
\begin{equation} \label{FFp} F_p(r) =
  \frac{r^{p-1}\int_{B_r}|\nabla u|^p\, dx}{\int_{\partial
      B_r}|u|^p\, dS},
\end{equation}
and a slight modification of \eqref{FFp}
\begin{equation} \label{FFp2} \widetilde F_p(r) =
  \frac{r\int_{B_r}|\nabla u|^p\, dx}{\int_{\partial
      B_r}|u|^p\, dS}.
\end{equation}
As for the frequency functions defined in \eqref{FF1} and \eqref{FF2},
it is easy to check that $F_p(r)$ satisfies the following scaling
property for each $\tau \in \R$, $\tau>0$,
\[
F_p^v(r)=F_p^{u}(\tau r),
\]
where $u$ is a solution to \eqref{eq:pLap} and $v(x)=u(\tau x)$. The
scaling property for the frequency function defined in \eqref{FFp2} is
slightly different and can be stated as follows
\[
\widetilde F_p^v(r)=\tau^{p-2}\widetilde F_p^u(\tau r).
\]
The results obtained in \cite{GraMa} were the following.

\begin{theorem} \label{thm:uniquelinear} Suppose $u\in
  W\loc^{1,p}(G)\cap C^2(G)$ is a solution to the \p-Laplace
  equation in $G$. Consider an affine function
\[
L(x) = l(x) + l_0,
\]
where $l_0\in\R$ and $$l(x) = \sum_{i=1}^n\alpha_ix_i$$ is not
identically zero. Then if $u(x)=L(x)$ in $B_r\subset G$,
$u(x)=L(x)$ for every $x\in G$.
\end{theorem}

\begin{remark}
  It can be shown that the difference $u-L$ satisfies a uniformly
  elliptic equation in divergence form with constant principal part
  coefficients, see equation (3.2) in \cite{GraMa}. It is standard,
  see e.g. \cite[Theorem 8.1, pp. 145--146]{ZachThoe}, that there
  exists a linear transformation of coordinates of the form
\[
\xi_i = \sum_{j=1}^n c_{ij}x_j, \quad i=1,\ldots,n,
\]
with nonsingular matrix $[c_{ij}]$, in such a way that equation (3.2)
in \cite{GraMa} can be reduced, in terms of the new coordinates
$\xi_1,\xi_2,\ldots,\xi_n$, to the canonical form
\eqref{eq:Pde}. Hence, in regard to nonlinear generalizations, it is
of interest to study the unique continuation principle for the
solutions to \eqref{eq:Pde}.
\end{remark}

The preceding theorem could be also stated as follows. Suppose
$u,\,v\in W\loc^{1,p}(G)\cap C^2(G)$ are two solutions to the
\p-Laplace equation in $G$. Assume further that $\nabla v\neq 0$ in
$G$. Then if $u(x)=v(x)$ in $B_r\subset G$, $u(x)=v(x)$ for every
$x\in G$.

\begin{theorem} \label{thm:weakDprop} Suppose $u\in C^1(G)$. Assume
  further that there exist two concentric balls $B_{r_b}\subset
  \overline{B}_{R_b}\subset G$ such that the frequency function
  $F_p(r)$ is defined, i.e., $I(r)>0$ for every $r\in(r_b,R_b]$, and
  moreover, $\|F_p\|_{L^\infty((r_b,R_b])}<\infty$. Then there exists
  some $r^\star\in(r_b,R_b]$ such that
\begin{equation} \label{eq:weakD1}
\int_{\partial B_{r_1}}|u|^p\, dS
  \leq 4\int_{\partial B_{r_2}}|u|^p\,
  dS,
\end{equation}
for every $r_1,\,r_2\in(r_b,r^\star]$. In particular, the following
weak doubling property is valid
\begin{equation} \label{eq:weakD2}
\int_{\partial B_{r^\star}}|u|^p\, dS
  \leq 4\int_{\partial B_r}|u|^p\,
  dS,
\end{equation}
for every $r\in(r_b,r^\star]$.
\end{theorem}

In the following we formulate a partial result on the unique
continuation problem for the \p-Laplace equation. It says that the
local boundedness of the frequency function implies the unique
continuation principle. In this respect the situation is similar to
the linear case $p=2$, and we thus generalize this phenomenon to every
$1<p<\infty$.

\begin{theorem} \label{thm:uniquecont} Suppose $u$ is a solution to
  the \p-Laplace equation in $G$. Consider arbitrary concentric balls
  $B_{r_b}\subset \overline{B}_{R_b}\subset G$. Assume the following:
  whenever $I(r)>0$ for every $r\in (r_b,R_b]$, then
  $\|F_p\|_{L^\infty((r_b,R_b])}<\infty$. Then the following unique
  continuation principle follows: If $u$ vanishes on some open ball in
  $G$, then $u$ is identically zero in $G$.
\end{theorem}

It remains an open problem whether the frequency function $F_p(r)$ is
locally bounded for the solutions to the \p-Laplace equation. Local
boundedness combined with the method of the present paper would solve
the unique continuation problem for equation \eqref{eq:pLap}.

\begin{remark}
  The corresponding divergence identity \eqref{eq:RNdiv} for solutions
  to the \p-Laplace equation is available, as well as the
  corresponding Rellich--Necas type formula, see e.g. \cite[Noether's
  theorem]{Evans} and \cite[Lemma 4.1]{HardtLin}, respectively.
\end{remark}

\end{document}